\documentclass{article}
\usepackage[utf8]{inputenc}
\usepackage[a4paper,total={5.8in,7.9in}]{geometry}
\usepackage{comment, amsmath, amssymb, amsthm, authblk, graphicx, mathtools}
\usepackage[mathlines]{lineno}

\newtheorem{lem}{Lemma}[section]
\newtheorem{cor}{Corollary}[section]
\newtheorem{pro}{Proposition}[section]
\newtheorem{teo}{Theorem}[section]
\theoremstyle{definition}
\newtheorem{defi}{Definition}[section]

\newcommand{\T}[1]{\Tilde{#1}}
\newcommand{\mc}[1]{\mathcal{#1}}
\newcommand{\mb}[1]{\mathbb{#1}}

\title{Equilibrium measures of manifolds without conjugate points having visibility covering}
\author[1]{Edhin Mamani\thanks{emamani@ufmg.br}}
\affil[1]{Instituto de Ciências Exatas ICEx, Universidade Federal de Minas Gerais, Av. Antônio Carlos 6627, Belo Horizonte, 31270-901, Brazil.}
\date{}
 
\begin{document}

%\linenumbers
\maketitle

\begin{abstract}
In this paper we study the equilibrium measures of geodesic flows of closed manifolds without conjugate points which have a visibility universal covering. Specifically, the uniqueness problem for Bowen potentials which are constants on some sets--intersection of horospheres--  and satisfy a weak pressure gap. Moreover, we study some ergodic properties of these measures such as the K-mixing property, weighted equidistribution of closed geodesics, the Gibbs property, large deviations and the entropy density of ergodic measures. Assuming, furthermore continuity of Green bundles, existence of a hyperbolic closed geodesic and a Gromov hyperbolic universal covering we prove that the above potentials always satisfy the weak pressure gap.    
\end{abstract}

\section{Introduction}
One of the most important examples of dynamical systems of geometric origin is the geodesic flow. These flows model the movement of a particle without any forces acting on it. It was the inspiration and the toy model for many developments in dynamical systems and ergodic theory. 

In the 1960's, Sinai, Ruelle, Bowen among others established a framework for the study of ergodic properties of dynamical systems called today Thermodynamics Formalism. Among other things, they proved the existence and uniqueness of the measure of maximal entropy for the geodesic flow of a compact manifold of variable negative curvature. The thermodynamic formalism provided a way to generalize the measure of maximal entropy, the so-called equilibrium measures or states. These measures are analogues of physical concepts coming from the Statistical Physics and Thermodynamics. In the 1980's, for compact manifolds of negative curvature, \cite{bowen72} proved the existence and uniqueness of equilibrium measures of Holder continuous potentials. 
%Moreover, \cite{x} proved finer statistical properties of these measures: the mixing and Bernoulli properties assuming X. 

The methods used by Bowen to show the uniqueness of equilibrium measures were based on the expansivity and the specification property of the flow. In 2014, Climenhaga and Thompsom extended these methods to non-uniform versions of expansivity and specification \cite{clim16}. Using this method Burns-Climenhaga-Fisher-Thompson proved the existence and uniqueness of equilibrium measures of Holder potentials for rank-1 manifolds of non-positive manifolds \cite{burns18}. On the other hand, using symbolic dynamics Lima and Poletti reproved the same result and showed the Bernoullicity of the measures \cite{lima25}. 

Another way to extend the analysis is to consider more general Riemannian metrics. The natural generalizations of Riemannian metrics of non-positive curvature are the metrics without focal points and metrics without conjugate points. These metrics have mostly negative curvature while allowing some zero and positive curvature. 

For the measure of maximal entropy Gelfert-Ruggiero, Climenhaga-Knieper-War and Mamani-Ruggiero \cite{gelf19,clim21,mam24} extended the uniqueness to the cases of no-focal points and no-conjugate points. For the case of equilibrium measures, there were extensions using the methods of Climenhaga and Thompson. 
%For compact manifolds without focal points, X showed the uniqueness of equilibrium measures of Holder potentials satisfying the pressure gap. |%otra referencia. 

On the other hand, the case of manifolds without conjugate points proved to be more complicated. For surfaces without conjugate points, Wu used the general results of Lima-Poletti which holds in a more general setting. In this context, assuming furthermore the continuity of Green bundles and a visibility universal covering Wu showed the existence and uniqueness of equilibrium measures of Holder potentials satisfying the generalized gap pressure \cite{wu24}. Following these works, we contribute to the case of manifolds without conjugate points as follows.
\begin{teo}\label{main}
Let $M$ be a closed manifold without conjugate points and with visibility universal covering. If $f:T_1M\to \mb{R}$ is a Bowen potential which is constant on classes and has the weak pressure gap then $f$ has a unique equilibrium measure $\mu$ and satisfies the following properties, 
\begin{enumerate}
    \item $\mu$ is K-mixing.
    \item The $f$-pressure of periodic prime orbits agree with topological pressure of $f$.
    \item Periodic prime orbits weighted by $f$ equidistribute to $\mu$.
    \item $\mu$ has the Gibbs property on expansive points.
    \item $\mu$ satisfies the level-1 large deviations principle for Bowen potentials which are constant on classes.
    \item The geodesic flow has a restricted entropy density of ergodic measures.
\end{enumerate}
\end{teo}
Comparing this theorem to Wu's result, note that we maintain the visibility hypothesis. This is because in higher dimensions the behavior can be very complicated and hence we assume some global geometry assumption in order to have some tools available. On the other hand, we remove the hypothesis of continuity of Green bundles assumed by Wu. However, we change the family of potentials for which the equilibrium measures are unique.

In the original Bowen work about the uniqueness of equilibrium measures, he was able to prove this property for potentials more general than Holder potentials, now called Bowen potentials. In manifolds of negative curvature, Holder potentials are Bowen potentials. In the non-uniformly case, the relation  between them is more diffuse but they both include the constant potentials, in particular the zero potential which corresponds to the measure of maximal entropy.

The other hypothesis imposed on the potentials considered in Theorem x seems to be a bit technical but it is common to the proof strategy that we follow. That is, we use some factor flow of the geodesic flow with uniqueness of equilibrium measures for Bowen potentials and then we somehow try to lift this property to the geodesic flow. Indeed, under our assumptions it is possible to build an equivalence relation on our phase space (unit tangent bundle) and hence define a quotient space. The geodesic flow induces a factor flow acting on the quotient space and this flow has the desired property. Since our potentials are defined on the unit tangent bundle, to induce potentials on the quotient space, we require that our potentials be constant on the equivalence classes. The use of a factor flow or factor map to show uniqueness of equilibrium measures is used in other contexts. In these contexts, it is also necessary to consider potentials constant on certain sets in order to apply the method.

Concerning the last hypothesis, for manifolds of nonpositive curvature, Burns-Climenhaga-Fisher-Thompson showed the pressure gap for potentials which are locally constant on a neighborhood of the singular set, i.e., the set where there is lack of hyperbolicity. In the setting of Theorem \ref{main}, we can prove a weak pressure gap in particular contexts as follows.
\begin{teo}
Let $M$ be a closed manifold without conjugate points and with Gromov hyperbolic universal covering. If $M$ has a hyperbolic closed geodesic and continuous Green bundles then every Bowen potential which is constant on classes, satisfy the weak pressure gap. Moreover, the unique equilibrium measure $\mu$ of $f$ is fully supported. In particular, $\mu$ satisfies all the properties of Theorem \ref{main}.
\end{teo}
%We highlight that condition 2 is particularly known in some previous results. In fact, considering the zero potential we recover Theorem x of x about uniqueness of measures of maximal entropy. On the other hand if we restrict potentials to be Bowen, Holder and constant on sets I, we recover part of Wu's Theorem without assuming the pressure gap. Clearly, this includes the setting of closed surfaces witout conjugate points and continuous Green bundles since these surfaces satisies the vi sibiity condition by x.   

%organizton of the paper.
\section{Preliminaries}
\subsection{Closed manifolds without conjugate points and horospheres}\label{m}
We introduce the general background and main notations we will use throughout the paper. Let $(M,g)$ be a $C^{\infty}$ closed (compact and boundaryless) Riemannian manifold, $TM$ be the tangent bundle and $T_1M$ be the unit tangent bundle. Consider the universal covering $\tilde{M}$ of $M$, the unit tangent bundle $T_1\T{M}$ of $\T{M}$ and the covering maps $\pi:\tilde{M}\to M$ and $d\pi:T_1\tilde{M}\to T_1M$. The covering $\tilde{M}$ is a complete Riemannian manifold with the pullback metric $\tilde{g}$. A manifold $M$ has no conjugate points if the exponential map $\exp_p$ is non-singular at every $p\in M$. In particular, $\exp_p$ is a covering map for every $p\in M$ (p. 151 of \cite{doca92}).

A geodesic is a smooth curve $\gamma\subset M$ satisfying $\nabla_{\dot{\gamma}}\dot{\gamma}=0$ where $\nabla$ is the covariant derivative induced by $g$. For every $\theta=(p,v)\in T_1M$, $\gamma_{\theta}$ is the unique geodesic satisfying the initial conditions $\gamma_{\theta}(0)=p$ and $\dot{\gamma}_{\theta}(0)=v$. The geodesic flow $\phi_t$ on $T_1M$ is defined by 
\[ \phi: \mathbb{R}\times T_1M\to T_1M, \qquad (t,\theta)\mapsto \phi_t(\theta)=\dot{\gamma}_{\theta}(t). \]

There exists a Riemannian metric $h$ on $TM$ induced by $g$ \cite{pater97}. We call $h$ Sasaki metric when restricted to $T_1M$ and call Sasaki distance $d_s$ the induced Riemannian distance on $T_1M$. Sasaki metric gives a decomposition of the tangent bundle of $T_1M$. For every $\theta\in T_1M$, denote by $G(\theta)\subset T_{\theta}T_1M$ the subspace tangent to the geodesic flow at $\theta$. Let $N(\theta)\subset T_{\theta}T_1M$ be the subspace orthogonal to $G(\theta)$ with respect to Sasaki metric. Thus, we have the following orthogonal decompositions
\[  T_{\theta}T_1M=N(\theta)\oplus G(\theta) \quad \text{ and }\quad N(\theta)=H(\theta)\oplus V(\theta).  \]
So, every $\xi\in N(\theta)$ has an orthogonal decomposition $\xi=(\xi_h,\xi_v)\in H(\theta)\oplus V(\theta)$.

Now, every $\theta\in T_1\T{M}$ has two important associated sets $\mc{\T{F}}^s(\theta)$ and $\mc{\T{F}}^u(\theta)$ included in $T_1\T{M}$ (\cite{esch77,pesin77}). With these sets we define the center stable and center unstable sets of $\theta$ by 
\[  \mathcal{\tilde{F}}^{cs}(\theta)=\bigcup_{t\in \mathbb{R}} \mathcal{\tilde{F}}^s(\phi_t(\theta))  \quad \text{ and }\quad \mathcal{\tilde{F}}^{cu}(\theta)=\bigcup_{t\in \mathbb{R}} \mathcal{\tilde{F}}^u(\phi_t(\theta)).  \]
Moreover, we collect them in families $\mathcal{\tilde{F}}^s=\{\mathcal{\tilde{F}}^s(\theta) \}_{\theta\in T_1\tilde{M}}$ and $\mathcal{\tilde{F}}^u=\{ \mathcal{\tilde{F}}^u(\theta) \}_{\theta\in T_1\tilde{M}}$.
%\[ \mathcal{\tilde{F}}^s=\{\mathcal{\tilde{F}}^s(\theta) \}_{\theta\in T_1\tilde{M}} \quad \text{ and } \quad \mathcal{\tilde{F}}^u=\{ \mathcal{\tilde{F}}^u(\theta) \}_{\theta\in T_1\tilde{M}}. \]
\begin{pro}[\cite{esch77,pesin77}]\label{horo}
Let $M$ be a closed $n$-manifold without conjugate points. Then, for every $\theta\in T_1\tilde{M}$, $\mathcal{\tilde{F}}^s(\theta)$ and $\mathcal{\tilde{F}}^u(\theta)$ are Lipschitz continuous embedded $(n-1)$-submanifolds of $T_1\T{M}$ invariant by the geodesic flow: $\phi_t(\mathcal{\tilde{F}}^*(\theta))=\mathcal{\tilde{F}}^*(\phi_t(\theta))$ for every $t\in \mb{R}$ and for $*=s,u$.
\end{pro}
We consider analogous objects in $T_1M$. For every $\theta\in T_1M$ and every lift $\T{\theta}\in T_1\T{M}$ of $\theta$ define 
\[  \mathcal{F}^*(\theta)=d\pi (\mathcal{\tilde{F}}^*(\tilde{\theta}))\subset T_1M \quad \text{ and }\quad \mathcal{F}^*=d\pi (\mathcal{\tilde{F}}^*), \quad *=s,u. \]

\subsection{Visibility manifolds}\label{v}
Let $M$ be a simply connected Riemannian manifold without conjugate points. We say that $M$ is a visibility manifold if for every $z\in M$ and every $\epsilon>0$ there exists $R(\epsilon, z)>0$ such that if $x,y\in M$ with $d(z,[x,y])>R(\epsilon,z)$ then $\sphericalangle_z(x,y)<\epsilon$, where $[x,y]$ is the unique geodesic segment joining $x$ to $y$ and $\sphericalangle_z(x,y)$ is the angle at $z$ formed by $[z,x]$ and $[z,y]$. If $R(\epsilon,z)$ does not depend on $z$ then $M$ is called a uniform visibility manifold. Let us note that a visibility universal covering of a compact manifold without conjugate points is actually a uniform visibility manifold \cite{eber72}. 

Visibility manifolds are important because of the good dynamical properties of their geodesic flows. Recall that a foliation is minimal if its leaves are dense.
\begin{teo}[\cite{eber72,eber73neg2}]\label{v1}
Let $M$ be a closed manifold without conjugate points, with visibility universal covering and with geodesic flow $\phi_t$. Then
\begin{enumerate}
    \item The families $\mathcal{F}^s$ and $\mathcal{F}^u$ are continuous minimal foliations of $T_1M$ invariant by $\phi_t$.
    \item The geodesic flow is topologically mixing.
\end{enumerate}
\end{teo}
Similarly, the families $\mc{\T{F}}^s$ and $\mc{\T{F}}^u$ are continuous invariant foliations of $T_1\tilde{M}$ usually called stable and unstable horospherical foliations. Hence $\mathcal{\tilde{F}}^s(\theta)$ and $\mathcal{\tilde{F}}^u(\theta)$ are respectively the stable and unstable horospherical leaf of $\theta\in T_1\T{M}$. 

We now consider intersections between horospherical leaves. For every $\theta\in T_1M$ and every lift $\T{\theta}\in T_1\T{M}$ of $\theta$, define the intersection sets $\mathcal{I}(\theta)=\mathcal{F}^s(\theta)\cap \mathcal{F}^u(\theta)$ and $\mathcal{\T{I}}(\T{\theta})=\mathcal{\T{F}}^s(\T{\theta})\cap \mathcal{\T{F}}^u(\T{\theta})$. The covering map $d\pi$ maps $\mathcal{\T{I}}(\T{\theta})$ onto $\mathcal{I}(\theta)$. We say that $\theta\in T_1M$ is expansive if $\mathcal{I}(\theta)=\{ \theta\}$ otherwise $\theta$ is non-expansive. The expansive set is denoted by
\[ \mathcal{R}_0=\{ \theta\in T_1M: \mathcal{I}(\theta)=\{ \theta\} \}.  \]
\begin{pro}[\cite{riff18}]\label{morse}
Let $M$ be a closed manifold without conjugate points and with visibility universal covering $\T{M}$. Then, there exists a uniform constant $Q(M)>0$ such that for every $\xi\in T_1\T{M}$, $\mc{\T{I}}(\xi)$ is a compact connected set with $Diam(\mc{\T{I}}(\xi))\leq Q(M)$.     
\end{pro}
In the visibility setting there is also a kind of ``local product". 
\begin{teo}[\cite{eber72}]\label{conec}
Let $M$ be a closed manifold without conjugate points and with visibility universal covering $\Tilde{M}$. Then, given $\xi,\xi'\in T_1\tilde{M}$ with $\xi\not\in \mathcal{\tilde{F}}^{cu}(\xi')$ there exist $\eta,\eta'\in T_1\tilde{M}$ with
\[ \mathcal{\tilde{F}}^s(\xi)\cap \mathcal{\tilde{F}}^{cu}(\xi')=\mathcal{\tilde{I}}(\eta) \quad \text{ and }\quad  \mathcal{\tilde{F}}^s(\xi')\cap \mathcal{\tilde{F}}^{cu}(\xi)=\mathcal{\Tilde{I}}(\eta').   \]
\end{teo}
These intersections are sometimes referred as the heteroclinic connections of the geodesic flow. 

\subsection{Jacobi fields and Green bundles}\label{jaco}
Let $(M,g)$ be a closed $n$-manifold without conjugate points, $\theta\in T_1M$ and $\gamma_{\theta}$ be the geodesic induced by $\theta$. A vector field $J$ along $\gamma_{\theta}$ is called a Jacobi field if $J$ satisfies the Jacobi equation
\[ J''(t)+R(\dot{\gamma}_{\theta}(t),J(t))\dot{\gamma}_{\theta}(t)=0, \]
where $R$ is the curvature tensor induced by $g$. Clearly a Jacobi field is totally determined by its initial conditions. Moreover, $J(t)$ is orthogonal to $\dot{\gamma}_{\theta}(t)$ for all $t\in \mathbb{R}$ if both $J(0)$ and $J'(0)$ are orthogonal to $\theta$. In this case, $J$ is called an orthogonal Jacobi field. We denote by $\mathcal{J}_{\theta}$ the vector subspace of orthogonal Jacobi fields along $\gamma_{\theta}$. For every $\xi\in T_{\theta}T_1M$, let $J_{\xi}$ be the Jacobi field on $\gamma_{\theta}$ satisfying $J_{\xi}(0)=\xi_h$ and $J'_{\xi}(0)=\xi_v$, where $\xi=(\xi_h,\xi_v)$ as in Subsection \ref{m}. By above, $J_{\xi}\in \mathcal{J}_{\theta}$ if and only if $\xi\in N(\theta)$. This can be extended to an isomorphism.
\begin{pro}[\cite{eber73.1,pater12}]\label{j1}
Let $M$ be a closed manifold without conjugate points. Then, for every $\theta\in T_1M$,
\begin{enumerate}
    \item The map $\xi\mapsto J_{\xi}$ is a linear isomorphism between $N(\theta)$ and $\mathcal{J}_{\theta}$.
    \item For every $\xi \in N(\theta)$ and every $t\in \mathbb{R}$, $d_{\theta}\phi_t(\xi)=(J_{\xi}(t),J'_{\xi}(t))$. 
\end{enumerate}
\end{pro}
In \cite{green58} Green defined 2 bundles existing always for compact manifolds without conjugate points. Given $\theta=(p,v)\in T_1M$, for every $T\in \mathbb{R}$ and every $w\in T_pM$ orthogonal to $v$, let $J_T$ be the Jacobi field satisfying $J_T(0)=w$ and $J_T(T)=0$. Green \cite{green58} showed that limits 
\begin{equation}\label{limi}
    J_w^s=\lim_{T\to \infty}J_T \quad \text{ and } \quad J_w^u=\lim_{T\to -\infty}J_T
\end{equation}
always exist and are Jacobi fields. They are called stable and unstable Jacobi field on $\gamma_{\theta}$, or simply Green fields. Green fields never vanish and form 2 vector subspaces of $\mc{J}_{\theta}$ denoted by $\mc{J}^s_{\theta}$ and $\mc{J}^u_{\theta}$. By Proposition \ref{j1}(1), $\mc{J}^s_{\theta}$ and $\mc{J}^u_{\theta}$ correspond to 2 vector subspaces $G^s(\theta)$ and $G^u(\theta)$ of $T_{\theta}T_1M$. The families $G^s=\bigcup_{\theta\in T_1M}G^s(\theta)$ and $G^u=\bigcup_{\theta\in T_1M}G^u(\theta)$ form 2 sub-bundles of the tangent bundle of $T_1M$, called stable and unstable Green bundle. Green bundles are just measurable and invariant by $d\phi_t$ \cite{eber73.1}. We say that Green bundles are continuous if $G^s(\theta)$ and $G^u(\theta)$ depend continuously on $\theta\in T_1M$. We have an improvement of Proposition \ref{j1}(2).
\begin{pro}[\cite{eber73.1}]\label{compa}
Let $M$ be a closed manifold without conjugate points. Then, there exists $B>0$ such that for every $\xi\in G^s\cup G^u$ and every $t\in \mb{R}$, $\|J_{\xi}(t)\|\leq \|d\phi_t(\xi)\|\leq B \|J_{\xi}(t)\|$.
\end{pro}
There is also a link with visibility manifolds. A combination of Proposition 3.6 of \cite{rugg21} and \cite{rugg94} express visibility in terms of Gromov hyperbolicity (see \cite{grom87} for a definition).
\begin{teo}\label{gromo}
Let $M$ be a closed manifold without conjugate points. If the universal covering $\T{M}$ is Gromov hyperbolic and Green bundles are continuous then $\T{M}$ is a visibility manifold.    
\end{teo}

\subsection{Topological and metric pressure}
We recall some basic notions of ergodic theory. Let $\Phi(\phi_t:X\to X)$ be a continuous flow acting on a compact metric space $X$. Given $\epsilon,T>0$ and $x\in X$, the $(\epsilon,T)$-dynamical ball is defined by 
\[ B(x,\epsilon,T)=\{y\in X: d(\phi_t(x),\phi_t(y))<\epsilon, \text{ for }t\in[0,T]  \}. \]
A subset $E\subset X$ is called $(\epsilon,T)$-spanning if $X=\bigcup_{x\in E}B(x,\epsilon,T)$. A continuous function $f:X\to \mb{R}$ is called potential on $X$. Given a potential $f$, varying over all possible $(\epsilon,T)$-spanning sets $E$, we define 
\[ S(f,\Phi,\epsilon,T)=\inf \left\{ \sum_{x\in E}e^{\int_0^Tf(\phi_t(x))dt}: E \text{ is a } (\epsilon,T)\text{-spanning set} \right\}.\]
The topological pressure of $f$ with respect to the flow $\Phi$ is given by 
\[ P(f,\Phi)=\lim_{\epsilon\to 0}\limsup_{T\to \infty}\frac{1}{T}\log S(f,\Phi,\epsilon,T). \]
Note that when $f\equiv 0$, the summation gives the cardinality of the $(\epsilon,T)$-spanning set $E$ hence $S(f,\Phi,\epsilon,T)$ is just the minimum cardinality of a $(\epsilon,T)$-spanning set. We thus get the topological entropy of the flow $h(\Phi)=P(0,\Phi)$. Moreover, for every $\Phi$-invariant measure $\mu$ on $X$, define its metric entropy by
\[ P_{\mu}(f,\Phi)=h_{\mu}(\phi_1)+\int_Xfd\mu, \]
where $h_{\mu}(\phi_1)$ is the metric entropy of the time-1 map $\phi_1$. For any potential on $X$ we have the following variational principle \cite{walt75},
\[  P(f,\Phi)=\sup_{\mu}P_{\mu}(f,\Phi),\]
where $\mu$ varies over all $\Phi$-invariant measures on $X$. A $\Phi$-invariant measure $\mu$ on $X$ is called equilibrium measure for $(f,\Phi)$ if $\mu$ attains the supremum.  

\section{The factor flow}\label{c26}

We give a brief sketch of the construction of a factor flow of the geodesic flow of a closed manifold $M$ without conjugate points and visibility universal covering $\T{M}$. The construction was carried originally by Gelfert-Ruggiero \cite{gelf19} and was extended to the higher dimensional setting by Mamani-Ruggiero \cite{mam24}. For every $\eta,\theta\in T_1M$, $\eta$ and $\theta$ are equivalent, 
\[  \eta\sim\theta \quad \text{ if and only if }\quad  \eta\in \mathcal{I}(\theta).\]
This is an equivalence relation that induces a quotient space $X$ and a quotient map $\chi:T_1M\to X$. We endow $X$ with the quotient topology and hence $\chi$ is continuous. For every $\theta \in T_1M$, we denote by $[\theta]=\chi(\theta)$ the equivalence class of $\theta$. Using the geodesic flow $\phi_t$ induced by $(M,g)$, we define a quotient flow $\Psi(\psi_t:X\to X)$ for every $\theta\in T_1M$ and every $t\in \mathbb{R}$ by 
\begin{equation}\label{for}
    \psi_t\circ \chi(\theta)=\chi\circ \phi_t(\theta).    
\end{equation}
We can make an analogous construction on the universal covering $\tilde{M}$: $\eta,\theta\in T_1\tilde{M}$ are equivalent if $\eta \in\mathcal{\tilde{I}}(\theta)=\mathcal{\tilde{F}}^s(\theta)\cap \mathcal{\tilde{F}}^u(\theta)\subset T_1\tilde{M}$. Similarly we get a quotient space $\tilde{X}$ endowed with the quotient topology, a quotient map $\Tilde{\chi}:T_1\tilde{M}\rightarrow \tilde{X}$ and a quotient flow $\T{\Psi}(\tilde{\psi}_t:\tilde{X}\rightarrow \Tilde{X})$. Since $T_1\tilde{M}$ is the universal covering of $T_1M$, the map $\Pi:\tilde{X}\to X$ satisfying $\chi\circ d\pi =\Pi\circ \tilde{\chi}$ is a well-defined covering map hence $\tilde{X}$ and $X$ are locally homeomorphic.
\begin{pro}[Theorem 1.1 of \cite{mam24}]
Let $M$ be a closed $n$-manifold without conjugate points and with visibility universal covering. Then, its geodesic flow is time-preserving semi-conjugate to a continuous flow $\Psi$ on a compact metric space $X$ of topological dimension at least $n-1$. 
\end{pro}
So, the factor flow $\Psi$ of the geodesic flow is just the quotient flow defined above. Among other dynamical properties of $\Psi$ we mention the following.
\begin{pro}[\cite{mam24}]\label{spe}
Let $M$ be a closed manifold without conjugate points and with visibility universal covering. Then, the factor flow $\Psi$ is topologically mixing and expansive. Moreover, $\Psi$ has a local product structure, the specification property and a unique measure of maximal entropy.
\end{pro}
A key tool to prove these properties was a special family of neighborhoods of fibers $\mc{I}(\xi)$.
\begin{pro}[Lemma 4.1 of \cite{mam24}]\label{que1}
Let $M$ be a closed manifold without conjugate points and with visibility universal covering. For every $\xi\in T_1M$ there exists a basis of neighborhoods $\mc{A}(\xi)$ of $\mc{I}(\xi)$ in $T_1M$ such that $\chi(\mc{A}(\xi))$ is a basis of neighborhoods of $\chi(\xi)$ in $X$. In particular, $\{ \chi(\mc{A}(\xi)): \xi\in T_1M \}$ is a basis for the quotient topology of $X$. 
\end{pro}
Given $\epsilon>0$ and a subset $A\subset T_1M$, we denote by $B(A,\epsilon)$ the $\epsilon$-neighborhood of $A$ with respect to Sasaki distance.
\begin{pro}\label{lev}
Let $M$ be a closed manifold without conjugate points and with visibility universal covering. For every $\epsilon>0$ there exists $\T{\epsilon}>0$ such that for every $[\xi],[\eta]\in X$ with $d([\xi],[\eta])<\T{\epsilon}$ there exists $\theta\in T_1M$ such that $\mc{I}(\xi),\mc{I}(\eta)\subset B(\mc{I}(\theta),\epsilon)$. In particular, if $\theta \in \mc{R}_0$ then for every $\xi'\in \mc{I}(\xi),\eta'\in \mc{I}(\eta)$ we have $d_s(\xi',\eta')<\epsilon$.
\end{pro}
\begin{proof}
Since $\mc{A}(\zeta)$ is a basis of neighborhoods of $\mc{I}(\zeta)$, choose an open set $A_{\zeta}\in \mc{A}(\zeta)$ such that $A_{\zeta}$ is contained in the $\frac{\epsilon}{2}$-neighborhood of $\mc{I}(\zeta)$, for every $\zeta\in T_1M$. Thus, $\mc{B}=\{\chi(A_{\zeta}):\zeta \in T_1M\}$ forms an open cover of the compact space $X$. Let $\T{\epsilon}>0$ be a Lebesgue number for $\mc{B}$ and let $[\xi],[\eta]\in X$ with $d([\xi],[\eta])<\T{\epsilon}$. It follows that there exists $\theta\in T_1M$ such that $[\xi],[\eta]\in \chi(A_{\theta}) \in \mathcal{B}$. Taking pre-images by $\chi$ we see that $\mc{I}(\xi)$ and $\mc{I}(\eta)$ are contained in the $\frac{\epsilon}{2}$-neighborhood of $\mc{I}(\theta)$ and hence in the corresponding $\epsilon$-neighborhood. If $\theta\in \mc{R}_0$ then $\mc{I}(\theta)=\{ \theta\}$ and $\mc{I}(\xi),\mc{I}(\eta)\subset B(\theta,\frac{\epsilon}{2})$. Therefore the triangle inequality provides the desired conclusion. 
\end{proof}

\section{Potentials having unique equilibrium measures}\label{pot}

Let $\phi_t:X\to X$ be a continuous flow acting on a compact metric space. A potential on $X$ is any continuous function $f:X\to \mb{R}$. We say that a potential $f$ is Bowen if there exist $C,\epsilon>0$ such that for every $s>0$ and every $x,x'\in X$ the following condition holds: 
\begin{equation}\label{bow}
\text{if }\quad d(\phi_t(x),\phi_t(x'))<\epsilon \quad\text{ for all } t\in [0,s] \quad \text{ then } \quad   \left|\int_0^sf(\phi_t(x))dt-\int_0^sf(\phi_t(x'))dt\right|\leq C.
\end{equation}
These potentials were introduced by Bowen in his work about uniqueness of equilibrium measures of continuous expansive flows with specification property \cite{bowenendo}. Now, consider a closed manifold $M$ without conjugate points and with visibility universal covering. Recall that $\chi:T_1M\to X$ is the quotient map induced by the equivalence relation defined in Section \ref{c26}. We observe that $\chi^{-1}(\chi(\xi))=\mc{I}(\xi)$ for every $\xi\in T_1M$ where $\mc{I}(\xi)=\mc{F}^s(\xi)\cap \mc{F}^u(\xi)$ (Subsection \ref{v}). Thus, the sets $\mc{I}(\xi)$ are exactly the equivalence classes of the equivalence relation in Section \ref{c26}. For brevity we refer to these equivalence classes $\mc{I}(\xi)$ simply as classes. 

We say that a continuous potential $f:T_1M\to \mb{R}$ is constant on classes if $f$ is constant when restricted to any class $\mc{I}(\xi)$. We remark that $f$ could have different values on different classes $\mc{I}(\xi)$ and $\mc{I}(\xi')$. In addition, $f$ has no restrictions on the expansive set $\mc{R}_0$, i.e., the set of singleton classes $\mc{I}(\xi)=\{\xi\}$ (Section \ref{v}).
\begin{defi}\label{hfam}
A continuous potential $f:T_1M\to\mb{R}$ belongs to the family $\mc{H}$ if $f$ is a Bowen potential which is constant on classes. 
\end{defi}
Since $f$ is constant on the equivalence classes $\mc{I}(\xi)$, it induces a unique well-defined function $\T{f}:X\to \mathbb{R}$ given by $\T{f}(\chi(\xi))=f(\xi)$ for every $\chi(\xi)\in X$. Clearly, this function satisfies $\T{f}\circ \chi=f$. Moreover, $\T{f}$ is continuous because $X$ is endowed with the quotient topology. We call $\T{f}$ the induced potential of $f$.
\begin{pro}\label{indu}
Let $M$ be a closed manifold without conjugate points and with visibility universal covering, and $\chi$ be the semi-conjugacy between the geodesic and quotient flow. For every $f\in \mc{H}$, the induced potential $\T{f}$ is Bowen and satisfies $\T{f}\circ \chi=f$.
\end{pro}
\begin{proof}
Since $f\in \mc{H}$ there exist $\epsilon,C>0$ satisfying Equation \eqref{bow}. For this $\epsilon>0$ choose $\T{\epsilon}>0$ from Proposition \ref{lev}. Suppose that $[\xi],[\eta]\in X$ with $d(\psi_t[\xi],\psi_t[\eta])<\T{\epsilon}$ for $t\in [0,s]$. Hence by Equation \eqref{for} we get $d([\phi_t(\xi)],[\phi_t(\eta)])<\T{\epsilon}$ for $t\in [0,s]$. Applying Proposition \ref{lev} for every $t\in [0,s]$ we get $\theta_t\in T_1M$ such that $\mc{I}(\phi_t(\xi)),\mc{I}(\phi_t(\eta))\subset B(\mc{I}(\theta_t),\epsilon)$. Since $D=\phi_{[0,s]}(\mc{I}(\xi))\cup \phi_{[0,s]}(\mc{I}(\eta))$ is a compact set, there exist finitely many open neighborhoods $B(\mc{I}(\theta_1),\epsilon),\ldots, B(\mc{I}(\theta_k),\epsilon)$ covering $D$. So, there exists a partition $0=t_0<t_1<\ldots<t_{k-1}<t_k=s$ such that
\begin{equation}\label{cub}
    \phi_{[t_{i-1},t_i]}(\mc{I}(\xi))\cup \phi_{[t_{i-1},t_i]}(\mc{I}(\eta))\subset B(\mc{I}(\theta_i),\epsilon) \quad \text{ for every} \quad i=1,\ldots,k.
\end{equation}
Thus, we can estimate the expression in Equation \ref{bow} in each partition interval. We have two cases. If $\theta_i\in T_1M\setminus \mc{R}_0$ then $\mc{I}(\theta_i)$ is not a singleton fiber and by definition $f$ is constant on $B(\mc{I}(\theta_i),\epsilon)$, the $\epsilon$-neighborhood of $\mc{I}(\theta_i)$. Using Equations \eqref{for} and \eqref{cub} we obtain
\begin{align*}
 \left|\int_{t_{i-1}}^{t_i}\left(\T{f}\circ\psi_t(\chi(\xi))-\T{f}\circ\psi_t(\chi(\eta))\right)dt\right|&= \left|\int_{t_{i-1}}^{t_i}\left(\T{f}\circ\chi(\phi_t(\xi'))-\T{f}\circ\chi(\phi_t(\eta'))\right)dt\right|\\
 &=\left|\int_{t_{i-1}}^{t_i}\left(f(\phi_t(\xi'))-f(\phi_t(\eta'))\right)dt\right|=0<C.   
\end{align*}
If $\theta_i\in \mc{R}_0$ then $\mc{I}(\theta_i)=\{ \theta_i\}$. Write $\xi'=\phi_{t_{i-1}}(\xi)$ and $\eta'=\phi_{t_{i-1}}(\eta)$. From Equation \eqref{cub} and Proposition \ref{lev} it follows that $d_s(\phi_t(\xi'),\phi_t(\eta'))<\epsilon$ for every $t\in [0,t_i-t_{i-1}]$. Applying the Bowen property of $f$ and Equation \eqref{for} we have     
\begin{align*}
 \left|\int_{t_{i-1}}^{t_i}\left(\T{f}\circ\psi_t(\chi(\xi))-\T{f}\circ\psi_t(\chi(\xi'))\right)dt\right|&= \left|\int_{t_{i-1}}^{t_i}\left(\T{f}\circ\chi(\phi_t(\xi))-\T{f}\circ\chi(\phi_t(\xi'))\right)dt\right|\\
 &=\left|\int_{t_{i-1}}^{t_i}\left(f(\phi_t(\xi))-f(\phi_t(\xi'))\right)dt\right|< C.   
\end{align*}
Summing up the estimates over all the partition intervals we conclude that $\T{f}$ is a Bowen potential. 

\end{proof}
We next recall a classical Franco's Theorem about equilibrium measures \cite{fran77}.
\begin{teo}\label{franco}
Let $\phi_t:X\to X$ be a continuous flow acting on a compact metric space. If $\phi_t$ is expansive and has the specification property then every Bowen potential has a unique equilibrium measure.
\end{teo}
\begin{cor}\label{frank}
Let $M$ be a closed manifold without conjugate points and with visibility universal covering. Then, for every $f\in \mc{H}$ the induced potential $\T{f}$ has a unique equilibrium measure.
\end{cor}
To end the section we mention a special disintegration of measures. Let $\mu$ be a $\Phi$-invariant measure on $T_1M$. Since fibers $\mc{I}(\xi)$ are closed sets by Corollary \ref{morse}, the family $\{ \mc{I}(\xi)\}_{\chi(\xi)\in X}$ forms a measurable partition of $T_1M$. The partition space is just the quotient space $X$ and the map sending $\xi\in T_1M$ to the partition element containing $\xi$, is just the quotient map $\chi:T_1M\to X$. By Rokhlin's Theorem there exist a system of conditional measures $\{\mu_x\}_{x\in X}$ on $T_1M$ and a measure $\chi_*\mu$ on the partition space $X$ such that for every $\mu$-measurable function $f:T_1M\to \mb{R}$, the map $x\mapsto \int fd\mu_x$ is $\chi_*\mu$-measurable and satisfies
    \[ \int_{T_1M}fd\mu=\int_X\left(\int_{T_1M}fd\mu_x\right)d(\chi_*\mu). \]
In particular, every $f\in \mc{H}$ is constant on each partition element $\mc{I}(\xi)$ with $x=\chi(\xi)\in X$. Thus,
\begin{equation}\label{des}
    \int_{T_1M}fd\mu_x=\int_{\mc{I}(\xi)}fd\mu = f(\xi)= \T{f}(x) \quad \text{ hence } \quad \int_{T_1M}fd\mu=\int_X \T{f}(x)d(\chi_*\mu).
\end{equation}

\section{Projection of equilibrium measures}
For the family of potentials $\mc{H}$, we will show that push-forward of equilibrium measures on $T_1M$ results in equilibrium measures on the quotient space $X$. 

Let us first recall the case of equilibrium measures for the zero potential, i.e., maximal measures. Clearly, the zero potential belongs to $\mc{H}$. Recall that $\Phi(\phi_t:T_1M\to T_1M)$ denotes the geodesic flow. 
\begin{pro}[\cite{mam24}]\label{sur}
Let $M$ be a closed manifold without conjugate points and with visibility universal covering, $\Psi$ be the quotient flow and $\chi$ be the quotient map. Then
\begin{enumerate}
    \item $h(\Phi)=h(\Psi)$ and $h_{\mu}(\phi_1)=h_{\chi_*\mu}(\psi_1)$ for any $\Phi$-invariant measure $\mu$ on $T_1M$.
    \item If $\mu$ is a maximal measure for $\Phi$ so is $\chi_*\mu$ for $\Psi$.
\end{enumerate}
\end{pro}
We denote by $\mc{M}(T_1M)$ and $\mc{M}(X)$ the spaces of Borel probability measures on $T_1M$ and $X$. 
\begin{pro}\label{exis}
Let $M$ be a closed manifold without conjugate points and with visibility universal covering, $(X,\Psi)$ be the quotient flow and $\chi$ be the quotient map. Given a $\Psi$-invariant measure $\nu$ on $X$ there exists a $\Phi$-invariant measure $\T{\nu}$ on $T_1M$ such that $\chi_*\T{\nu}=\nu$.
\end{pro}
\begin{proof}
First, recall that $\chi_*:\mc{M}(T_1M)\to \mc{M}(X)$ is an affine continuous surjective map because $\chi$ is continuous and surjective. Thus, there exists $\sigma\in \mc{M}(T_1M)$ such that $\chi_*\sigma=\nu$ but which may not be $\Phi$-invariant. For every $T>0$, define the average measure
\[  \nu_T=\frac{1}{T}\int_0^T(\phi_t)_*\sigma dt. \]
Clearly, $\nu_T$ is a Borel probability measure on $T_1M$. Using Equation \eqref{for} and $\Psi$-invariance of $\nu$, for every $T>0$ we obtain
\[  \chi_*\nu_T = \frac{1}{T}\int_0^T (\chi \circ \phi_t)_*\sigma dt = \frac{1}{T}\int_0^T (\psi_t\circ \chi)_*\sigma dt = \frac{1}{T}\int_0^T (\psi_t)_*\circ \chi_*\sigma dt=\frac{1}{T}\int_0^T (\psi_t)_*\nu dt = \nu.\]
Observe that preimage $(\chi_*)^{-1}\nu$ is a non-empty and compact set in the weak$^*$ topology. So, there exists an accumulation point of $(\nu_T)_{T>0}$ denoted by $\T{\nu}$. Since $\nu_T\in (\chi_*)^{-1}\nu$ for every $T>0$, compactness implies that $\T{\nu}\in (\chi_*)^{-1}\nu$ hence $\chi_*\T{\nu}=\nu$. Finally, we show the $\Phi$-invariance of $\T{\nu}$:
\begin{align*}
    |(\phi_s)_*\T{\nu}-\T{\nu}|&= \lim_{T_n}|(\phi_s)_*\nu_{T_n}-\nu_{T_n}|= \lim_{T_n}\left| \frac{1}{T_n}\int_0^{T_n}(\phi_{t+s})_*\sigma dt-\frac{1}{T_n}\int_0^{T_n}(\phi_{t})_*\sigma dt\right| \\
    &= \lim_{T_n}\left|\frac{1}{T_n}\int_s^{T_n+s}(\phi_{\tau})_*\sigma d\tau-\frac{1}{T_n}\int_0^{T_n}(\phi_{t})_*\sigma dt\right|= \lim_{T_n}\left|\frac{1}{T_n}\int_{[0,s]\cup [T_n,T_n+s]} (\phi_t)_*\sigma dt\right|.
\end{align*}
So, for every Borel set $A\subset T_1M$ we have
\[ |(\phi_s)_*\T{\nu}(A)-\T{\nu}(A)| = \lim_{T_n}\left|\frac{1}{T_n}\int_{[0,s]\cup [T_n,T_n+s]} \sigma(\phi_{t}^{-1}(A)) dt\right| \leq \lim_{T_n\to \infty}\frac{2s}{T_n} = 0. \]
\end{proof}

\begin{pro}\label{proy}
Let $M$ be a closed manifold without conjugate points and with visibility universal covering, $(X,\Psi)$ be the quotient flow, $\chi$ be the quotient map, $f\in \mc{H}$ and $\T{f}$ be the induced potential. If $\mu$ is an equilibrium measure for $f$ then so is $\chi_*\mu$ for $\T{f}$. 
\end{pro}

\begin{proof}
Since $\chi$ is a semi-conjugacy it follows that $\chi_*\mu$ is a $\Psi$-invariant measure on $X$. Using Proposition \ref{sur}(1) and Equation \eqref{des} we see that metric pressures of $\mu$ and $\chi_*\mu$ agree:
\begin{equation}\label{pr}
P_{\mu}(f,\Phi)=h_{\mu}(\phi_1)+\int_{T_1M} fd\mu=h_{\chi_*\mu}(\psi_1)+\int_X \T{f} d(\chi_*\mu)= P_{\chi_*\mu}(\T{f},\Psi)
\end{equation}
Moreover, topological pressures of $f$ and $\T{f}$ also agree. Equation \eqref{pr} gives the direct inequality:
\begin{equation}\label{pres}
    P(f,\Phi)=\sup_{\mu\in \mc{M}(\Phi)}P_{\mu}(f,\Phi) = \sup_{\mu\in \mc{M}(\Phi)}P_{\chi_*\mu}(\T{f},\Psi)\leq \sup_{\nu\in\mc{M}(\Psi)}P_{\nu}(\T{f},\Psi)=P(\T{f},\Psi).
\end{equation}
For the reverse inequality, Proposition \ref{exis} says that for every $\nu\in\mc{M}(\Psi)$ there exists $\T{\nu}\in \mc{M}(\Phi)$ such that $\chi_*\T{\nu}=\nu$ and hence $P_{\T{\nu}}(f,\Phi)=P_{\nu}(\T{f},\Psi)$. From this we see that 
\[  P(\T{f},\Psi)=\sup_{\nu\in\mc{M}(\Psi)}P_{\nu}(\T{f},\Psi) = \sup_{\T{\nu}\in \mc{M}(\Phi)}P_{\T{\nu}}(f,\Psi)\leq \sup_{\mu\in \mc{M}(\Phi)}P_{\mu}(f,\Phi)=P(f,\Phi).\]
We thus obtain 
\begin{equation}\label{preso}
    P(f,\Phi)=P(\T{f},\Psi).
\end{equation} 
By hypothesis and Equation \eqref{pr} we have $P_{\chi_*\mu}(\T{f},\Psi)=P_{\mu}(f,\Phi)=P(f,\Phi)=P(\T{f},\Psi)$. Therefore $\chi_*\mu$ is an equilibrium measure for $\T{f}$.
\end{proof}

\section{Uniqueness of equilibrium measures}\label{uni}

In this section, under some particular conditions we prove uniqueness of equilibrium measures for the family of potentials $\mc{H}$ considered in Section \ref{pot}.

Let $M$ be a closed manifold without conjugate points, $\Phi(\phi_t:T_1M\to T_1M)$ be its geodesic flow and $f:T_1M\to \mb{R}$ be a potential. Recall that $\mc{R}_0$ is the set of expansive points. We say that $f$ has the pressure-gap if 
\[ \sup_{\mu} \{P_{\mu}(f,\Phi):\mu \text{ is supported on } T_1M\setminus \mc{R}_0 \}<P(f,\Phi), \]
where $P_{\mu}(f,\Phi)$ and $P(f,\Phi)$ are the metric and topological pressure respectively. For short we denote by $\mc{R}_0^c$ the set $T_1M\setminus \mc{R}_0$, i.e., the non-expansive set.
\begin{lem}
Let $M$ be a closed manifold without conjugate points and with visibility universal covering. If $f\in \mc{H}$ has the pressure-gap then $\mc{R}_0^c$ cannot support an equilibrium measure of $f$.
\end{lem} 
\begin{proof}
Suppose there exists an equilibrium measure $\mu$ of $f\in \mc{H}$ supported on $\mc{R}_0^c$. Thus, $P(f,\Phi)=P_{\mu}(f,\Phi)\leq \sup_{\mu} \{P_{\mu}(f,\Phi):Supp(\mu)\subset \mc{R}^c_0 \}$ and hence $f$ has no pressure-gap.
\end{proof}
Therefore, the statement of the conclusion of this lemma is weaker than the pressure-gap of $f$. Thus, we say that $f$ has the weak pressure-gap if $\mc{R}_0^c$ cannot support an equilibrium measure of $f$. We will work with this weaker hypothesis instead of the classical one.
\begin{teo}\label{toe}
    Let $M$ be a closed manifold without conjugate points and with visibility universal covering. If $f\in \mc{H}$ has the weak pressure gap then $f$ has a unique equilibrium measure $\mu$. In particular, $\mu$ is ergodic. 
\end{teo}
\begin{proof}
By hypothesis the geodesic flow is $C^{\infty}$ and hence $f$ always has an equilibrium measure. For the uniqueness, let $\mu\in \mc{M}(\Phi)$ be any equilibrium measure for $f$ and $\mu_e$ be any ergodic component of $\mu$. It follows that $\mu_e$ is an equilibrium measure for $f$ as well. As $\mc{R}_0$ is $\Phi$-invariant, ergodicity provides that $\mu_e$ is supported either on $\mc{R}_0$ or on $\mc{R}_0^c$. The weak pressure-gap of $f$ implies that $\mu_e$ is supported on $\mc{R}_0$. Since all ergodic components of $\mu$ are supported on $\mc{R}_0$ so does $\mu$. From Section \ref{pot}, we consider the induced potential $\T{f}:X\to \mb{R}$ of $f$. By Corollary \ref{frank}, $\T{f}$ has a unique equilibrium measure $\nu$ with respect to the quotient flow $\Psi$. This together with Proposition \ref{proy} implies that $\mu$ projects on $\nu$: $\chi_*\mu=\nu$. Since $\mu$ is supported on $\mc{R}_0$ and $\chi$ is a bijection on $\mc{R}_0$, it follows that $\mu$ is completely determined by $\nu$. Therefore, every equilibrium measure of $f$ has the same form and the uniqueness follows.
\end{proof}
Next, we give a criterion guaranteeing full support for any equilibrium measure of $f\in \mc{H}$. First, recall the quotient flow $\Psi(\psi_t: X\to X)$ acting on the quotient space $X$ and the quotient map $\chi:T_1M\to X$. A subset $A\subset T_1M$ is saturated with respect to $\chi$ if $A=\chi^{-1}\circ \chi(A)$. The following result is a straightforward consequence of Lemma 3.1 of \cite{mam23}.  
\begin{lem}\label{sat}
Let $M$ be a closed manifold without conjugate points and with visibility universal covering. Then, for every open set $U\subset T_1M$ there exists an open saturated set $U'\subset U$. In particular, $\chi(U')$ is an open set in $X$.
\end{lem}
The following is a restatement of Proposition 7.3.15 of \cite{fish19} in our context. 
\begin{lem}\label{gibs}
Let $\Psi:X\to X$ be a continuous flow on a compact metric space $X$, $f:X\to \mathbb{R}$ be a Bowen potential and $\nu$ be an equilibrium measure of $f$. If $\Psi$ is expansive and has the specification property then $\nu$ has the Gibbs property for $f$ with constant $P(f,\Psi)$. In particular, $\nu$ has full support.
\end{lem}
The following is an extension of Lemma 10.2 of \cite{mam24} to the case of equilibrium measures.
\begin{lem}\label{supp}
Let $M$ be a closed manifold without conjugate points and with visibility universal covering. If $\mc{R}_0$ is dense then any equilibrium measure of $f\in \mc{H}$ has full support.
\end{lem}

\begin{proof}
Let $U\subset T_1M$ be an open set and $\mu$ be an equilibrium measure of $f$. By Lemma \ref{sat} there is an open saturated set $U'\subset U$. We observe that $U'$ could be empty in general. However, density of $\mc{R}_0$ ensures that $U'$ is non-empty and hence $\chi(U')\subset X$ is a non-empty open set. From Proposition \ref{proy} it follows that $\chi_*\mu$ is an equilibrium measure of $\T{f}$. Proposition \ref{spe} and Lemma \ref{gibs} imply that $\chi_*\mu$ has full support and hence $\chi_*\mu(\chi(U'))>0$. Since $U'$ is saturated, the conclusion follows from $\mu(U)\geq \mu(U')=\mu(\chi^{-1}\circ \chi (U'))=\chi_*\mu(\chi(U'))>0$.
\end{proof}

\begin{teo}\label{int}
Let $M$ be a closed manifold without conjugate points and with visibility universal covering. If $Int(\mc{R}_0)\neq \emptyset$ then every $f\in \mc{H}$ has a unique equilibrium measure. This measure is ergodic and fully supported.    
\end{teo}

\begin{proof}
We first show the density of $\mc{R}_0$. Let $W$ be an open set contained in $\mc{R}_0$. By Theorem \ref{v1}(2) there exists a dense orbit $\beta$ of the geodesic flow $\Phi$. Thus $\beta$ intersects $W\subset \mc{R}_0$ and hence $\beta\subset \mc{R}_0$ by the $\Phi$-invariance of $\mc{R}_0$. Since $\beta$ is dense so is $\mc{R}_0$. Now, let $\mu$ be an equilibrium measure of $f$ and $\mu_e$ be an ergodic component of $\mu$. Clearly, $\mu_e$ is also an equilibrium measure of $f$ and moreover $\mu_e$ has full support by Lemma \ref{supp}. In particular, $\mu_e(W)>0$ and hence $\mu_e(\mc{R}_0)=1$ by ergodicity of $\mu_e$. Since all ergodic components of $\mu$ are supported on $\mc{R}_0$ so does $\mu$. Thus $f$ has the weak pressure-gap and the result follows from Theorem \ref{toe} and Lemma \ref{supp}.
\end{proof}

\begin{cor}\label{gramo}
Let $M$ be a closed manifold without conjugate points. If the universal covering $\T{M}$ is Gromov hyperbolic, $M$ has continuous Green bundles and a hyperbolic closed geodesic then every $f\in \mc{H}$ has a unique equilibrium measure. This measure is ergodic and fully supported.    
\end{cor}

\begin{proof}
By Theorem \ref{gromo}, continuity of Green bundles and Gromov hyperbolicity imply that $\T{M}$ is a visibility manifold. This together with our hypothesis show that $\mc{R}_1$ is an open set contained in $\mc{R}_0$ by Proposition x. Therefore, the result follows from Theorem \ref{int}. 
\end{proof}

\section{K-mixing Property}
Let us recall the definition of \( K \)-mixing property in the discrete setting. Let $(X,\mc{B})$ be a measurable space, $\mu$ be a measure on $(X,\mc{B})$ and $g:X\to X$ be an invertible map preserving $\mu$. The system \( (X, \mathcal{B}, g, \mu) \) is said to be \( K \)-mixing if for any sets \( A_0, A_1, \dots, A_r \in \mathcal{B} \) for \( r \geq 0 \) it holds
\[
\lim_{n \to \infty}\sup_{B\in \mc{C}^r_n} \left| \mu(A_0 \cap B) - \mu(A_0)\mu(B) \right| = 0,
\]
where \( \mc{C}^r_n \) is the minimal \( \sigma \)-algebra generated by \( g^k(A_i) \) for \( 1 \leq i \leq r \) and \( k \geq n \). We see that \( K \)-mixing property implies mixing of all orders. Thus, \( K \)-mixing is interpreted as a strong mixing property of $\mu$. A measure-preserving flow $(X,\mc{B},\Psi,\mu)$ is $K$-mixing if for every $t\neq 0$ the discrete system $(X,\mc{B},\psi_t,\mu)$ is $K$-mixing. In fact, by Gurevic's work \cite{gure} we only need to verify that a single time-$t$ map is $K$-mixing.

We will state a criterion to get the $K$-mixing property in the context of flows due to Call and Thompson \cite{call}. This criterion is based on a previous work of Ledrappier in the discrete setting \cite{led-criterion}. Let $\Psi(\psi_t:X\to X)$ be a continuous flow on a compact metric space and $f$ be a continuous potential on $X$. We consider the self-product flow $\Psi\times \Psi: X\times X\to X\times X$ and the sum potential $F:X\times X\to \mb{R}$ defined by $F(x,x')=f(x)+f(x')$. We define the set of non-expansive points of $X$ at scale $\epsilon>0$ by
\[  NE(\epsilon)= \{ x\in X: \Gamma_{\epsilon}(x)\not\in \psi_{[-s,s]}(x), \text{ for every } s>0\}, \]
where $\Gamma_{\epsilon}(x)=\{y\in X:d(\psi_t(x),\psi_t(y))<\epsilon, \text{ for every } t\in \mb{R}  \}$. A $\Psi$-invariant measure $\mu$ is called almost expansive if there exists $\epsilon>0$ such that $\mu(NE(\epsilon))=0$. Moreover, the set of product non-expansive points of $X\times X$ at scale $\epsilon>0$ is given by
\[  NE^{\times}(\epsilon)= \{ (x,y)\in X\times X: \Gamma_{\epsilon}(x,y)\not\in \psi_{[-s,s]}(x)\times \psi_{[-s,s]}(y), \text{ for every } s>0\}. \]
A $\Psi\times \Psi$-invariant measure $\nu$ is product expansive if there exists $\epsilon>0$ such that $\nu(NE^{\times}(\epsilon))=0$.
\begin{pro}[\cite{call}]\label{calo}
    Let $(X,\Psi)$ be a continuous entropy expansive flow on a compact metric space and $f:X \to [0, \infty)$ be a continuous potential. Assume that any equilibrium measure of $f$ is almost expansive and any equilibrium measure of $F$ is product expansive. Suppose that $X \times [0, \infty)$ has a $\lambda$-decomposition $(\mathcal{P},\mathcal{G},\mathcal{S})$ satisfying:
    \begin{enumerate}
        \item $\mathcal{G}$ has specification at all scales.
        \item $f$ has the Bowen property on $\mathcal{G}$.
        \item $P(\mathcal{P} \cup \mathcal{S}, f) < P(f)$, and the corresponding $\T{\lambda}$-decomposition $(\mathcal{\T{P}},\mathcal{\T{G}},\mathcal{\T{S}})$ for $(X \times X, \Psi \times \Psi)$ satisfies
        \item $P(\mathcal{\T{P}} \cup \mathcal{\T{S}}, F) < P(F) = 2P(f)$.
    \end{enumerate}
    Then $(X\times X,\Psi\times \Psi,F)$ has a unique equilibrium measure and thus the unique equilibrium measure for $(X,\Psi,f)$ is $K$-mixing.
\end{pro}
\begin{lem}\label{produ}
    Let $(X,\Psi)$ be a continuous expansive flow on a compact metric space. Then every $\Psi\times\Psi$-invariant measure $\mu$ on $X\times X$ is product expansive.
\end{lem}
\begin{proof}
    By Lemma 4.4 of \cite{call} we obtain $NE^{\times}(\epsilon)=(X\times NE(\epsilon))\cup (NE(\epsilon)\times X)$ for every $\epsilon>0$. Since $\Psi$ is expansive, $NE(\epsilon)=\emptyset$ for some $\epsilon>0$. Therefore $NE^{\times}(\epsilon)=\emptyset$ and hence $\mu(NE^{\times}(\epsilon))=0$.  
\end{proof}
\begin{pro}
    Let $M$ be a closed manifold without conjugate points and with visibility universal covering. If $f\in \mc{H}$ has the weak pressure-gap then the unique equilibrium measure of $f$ is $K$-mixing.
\end{pro}
\begin{proof}
 By Proposition \ref{spe}, the quotient flow $\Psi$ is expansive and thus entropy expansive. Moreover, every $\Psi$-invariant measure on $X$ is almost expansive. Lemma \ref{produ} implies that any $\Psi\times \Psi$-invariant measure on $X\times X$ is product expansive. Proposition \ref{spe} asserts that $\Psi$ has the specification property and therefore so does $\Psi\times\Psi$ by Lemma 3.1 of \cite{call}. For this reason we can take the $\lambda$-decomposition $(\mathcal{P},\mathcal{G},\mathcal{S})$ of $X\times [0,\infty]$ and $\T{\lambda}$-decomposition $(\mathcal{\T{P}},\mathcal{\T{G}},\mathcal{\T{S}})$ of $X\times X\times [0,\infty]$ to be trivial, i.e., $\mathcal{P}=\mathcal{S}=\mathcal{\T{P}}=\mathcal{\T{S}}=\emptyset$ with $\mc{G}$ and $\mc{\T{G}}$ collecting all possible orbit segments of $\Psi$ and $\Psi\times \Psi$. Since the induced potential $\T{f}$ on $X$ is Bowen by Proposition \ref{indu}, all hypothesis of Proposition \ref{calo} are satisfied. Therefore the sum potential $\T{F}(x,x')=\T{f}(x)+\T{f}(x')$ has a unique equilibrium measure and the unique equilibrium measure $\nu$ of $\T{f}$ is $K$-mixing. From Theorem \ref{toe}, let $\mu$ be the unique equilibrium measure of $f$, which satisfies $\chi_*\mu=\nu$. Hence $(T_1M,\mu)$ and $(X,\nu)$ are isomorphic because $\mu(\mc{R}^c_0)=\nu(\chi(\mc{R}^c_0))=0$ and $\chi$ is a bijection on $\mc{R}_0$. Therefore, since $\nu$ is $K$-mixing so is $\mu$.    
\end{proof}
Lemma \ref{produ} and part of the above proof also show the $K$-mixing property for the equilibrium measures considered in Bowen-Franco's Theorem \ref{franco}.
\begin{teo}
    Let $(X,\Psi)$ be a continuous expansive flow on a compact metric space and $f$ be a Bowen potential on $X$. If $\Psi$ has the specification property then $f$ has a unique equilibrium measure, which is $K$-mixing. 
\end{teo}

\section{Weighted equidistribution of periodic orbits}

Let $\Psi=(\psi_t)$ be a continuous flow acting on a compact metric space $X$. For every $\Psi$-periodic orbit $\gamma$ in $X$ we denote its minimal period by $|\gamma|$. For every $T>0$ and every $\epsilon>0$ small enough we denote by $\mc{O}_{T,\epsilon}$ the set of $\Psi$-periodic orbits $\gamma$ with $|\gamma|\in[T-\epsilon,T+\epsilon]$. Given a continuous potential $f:X\to \mb{R}$ and a $\Psi$-periodic orbit $\gamma$ in $X$ we denote the integral of $f$ around $\gamma$ by 
\[  f(\gamma)=\frac{1}{|\gamma|}\int_0^{|\gamma|}f(\psi_t(x_{\gamma}))dt \quad \text{ for some }\quad x_{\gamma}\in \gamma.  \]
For a fixed small enough $\epsilon>0$ let us consider
\begin{equation}\label{lam}
    \limsup_{T\to \infty}\frac{1}{T}\log \Lambda(f,\Psi,T,\epsilon)  \quad \text{ with }\quad \Lambda(f,\Psi,T,\epsilon)=\sum_{\gamma\in\mc{O}_{T,\epsilon}}e^{f(\gamma)}.
\end{equation}
If this $\limsup$ does not depend upon $\epsilon$ then we denote it by $\overline{P}^*(f,\Psi)$ and call it the upper $f$-pressure of periodic orbits. Analogously we define the lower $f$-pressure of periodic orbits $\underline{P}^*(f,\Psi)$ using $\liminf$ instead of $\limsup$. Whenever these two pressures exist and agree, we denote the common value by $P^*(f,\Psi)$ and call it the $f$-pressure of periodic orbits.

We can use $f$ and periodic orbits to construct $\Psi$-invariant measures by the averaging. For every $T>0$ and $\epsilon>0$ small enough we define 
\[ \mu_{T,\epsilon}=\frac{1}{\Lambda(f,\Psi,T,\epsilon)}\sum_{\gamma\in \mc{O}_{T,\epsilon}}e^{f(\gamma)}\mu_{\gamma}, \]
where $\mu_{\gamma}$ is the unique $\Psi$-invariant measure on $X$ supported on $\gamma$. This measure is a well-defined $\Psi$-invariant probability measure on $X$ for example when $\Psi$ is expansive because $|\mc{O}_{T,\epsilon}|$ and $\Lambda(f,\Psi,T,\epsilon)$ are finite. Given a $\Psi$-invariant measure $\mu$ on $X$, periodic orbits weighted by $f$ equidistribute to $\mu$ if for every $\epsilon>0$ small enough, $\mu_{T,\epsilon}\to \mu$ in the weak$^*$ topology as $T\to \infty$.

For geodesic flows of manifolds of negative curvature, Bowen showed that periodic orbits (for the zero potential) equidistribute to the unique measure of maximal entropy \cite{bowenequi}. This was extended to hyperbolic flows by Parry, showing that periodic orbits weighted by a Holder continuous potential $f$, equidistribute to the unique equilibrium measure of $f$ \cite{parry06}. These cases deal with expansive flows and hence there exist finitely many periodic orbits with period in $[T-\epsilon,T+\epsilon]$ for $\epsilon>0$ small enough. For manifolds of non-positive curvature, there may exist closed flat strips of zero curvature, i.e., infinite periodic orbits with same period. 

For closed manifolds without conjugate points and with visibility universal covering, recall that strips $S(\eta)$ are built of equivalence classes $\mc{I}(\eta)$ (defined in Section \ref{c26}),
\[ S(\eta)=\bigcup_{t\in \mb{R}}\phi_t(\mc{I(\eta)}) \quad\text{ for some }\quad \eta \in T_1M. \]
The strip is periodic if there is some $T>0$ such that $\phi_T(\mc{I}(\eta))=\mc{I}(\eta)$. If $\mc{\mc{I}(\eta)}=\{\eta\}$ then $S(\eta)$ is reduced to the single orbit $\phi_{\mb{R}}(\eta)$ and a periodic strip is reduced to a single periodic orbit. In our setting, a periodic strip is not necessarily made up solely of periodic orbits. However, each periodic strip has at least one periodic orbit, and all of its periodic orbits have the same period. Thus, for each periodic strip $S(\eta)$, we choose any one of its periodic orbits and call it a prime periodic orbit of $S(\eta)$. For every $T>0$ and $\epsilon>0$ small enough we denote by $\mc{O}'_{T,\epsilon}$ the set of $\Phi$-periodic prime orbits $\gamma$ with $|\gamma|\in[T-\epsilon,T+\epsilon]$.

For potentials $f\in \mc{H}$ which are constant on classes $\mc{I}(\eta)$, using $\mc{O}'_{T,\epsilon}$ we define an analogue $f$-pressure to above. For every $T>0$ and $\epsilon>0$ small enough, the $f$-pressure of periodic prime orbits is defined by
\[  P'(f,\Phi)=\lim_{T\to \infty}\frac{1}{T}\log \Lambda'(f,\Phi,T,\epsilon) \quad \text{ with } \quad  \Lambda'(f,\Phi,T,\epsilon)=\sum_{\gamma\in\mc{O}'_{T,\epsilon}}e^{f(\gamma)}<\infty,   \]
where the limit exists and does not depend upon $\epsilon$. 
%definir los tipos de presiones y referenciar a Katrin
\begin{lem}
Let $M$ be a closed manifold without conjugate points and with visibility universal covering. If $f\in \mc{H}$ then $P'(f,\Phi)=P(f,\Phi)<\infty$.
\end{lem}

\begin{proof}
Let $\tilde{f}$ be the induced potential on $X$ and $\tilde{\epsilon}>0$ the its Bowen constant given by Proposition \ref{indu}. For every $T>0$ and every $\epsilon<\tilde{\epsilon}$ small enough let $\mc{O}_{T,\epsilon}$ and $\mc{O}'_{T,\epsilon}$ be respectively the sets of $\Psi$-periodic orbits and $\Phi$-periodic prime orbits with period in $[T-\epsilon,T+\epsilon]$ as explained above. Since the semi-conjugacy $\chi$ between $\Phi$ and $\Psi$ is time-preserving it follows that $\chi(\mc{O}'_{T,\epsilon})\subset \mc{O}_{T,\epsilon}$. The reverse inclusion follows from the surjectivity of $\chi$ and the above definition of $\Phi$-periodic prime orbits. We thus obtain $\chi(\mc{O}'_{T,\epsilon})=\mc{O}_{T,\epsilon}$ and $|\mc{O}'_{T,\epsilon}|=|\mc{O}_{T,\epsilon}|<\infty$ by expansivity of $\Psi$ given in Proposition \ref{spe}. On the other hand, from the semi-conjugacy Equation \eqref{for} and Proposition \ref{indu}, for every $\gamma\in \mc{O}'_{T,\epsilon}$ we have 
\begin{equation}\label{wei}
    f(\gamma)=\int_0^{|\gamma|}f(\phi_t(\eta_{\gamma}))dt= \int_0^{|\chi(\gamma)|}\tilde{f}(\psi_t(x_{\chi(\gamma)}))dt= \tilde{f}(\chi(\gamma)),
\end{equation}
for some $\eta_{\gamma}\in \gamma$ and $x_{\chi(\gamma)}\in \chi(\gamma)$ where $|\chi(\gamma)|=|\gamma|$ and $\chi(\gamma)\in\mc{O}_{T,\epsilon}$. Using all together gives
\begin{equation}\label{expa}
    \Lambda'(f,\Phi,T,\epsilon)=\sum_{\gamma\in \mc{O}'_{T,\epsilon}}e^{f(\gamma)}=\sum_{\chi(\gamma)\in \chi(\mc{O}'_{T,\epsilon})}e^{\tilde{f}(\chi(\gamma))}=\sum_{\T{\gamma}\in \mc{O}_{T,\epsilon}}e^{\tilde{f}(\T{\gamma})}= \Lambda(\tilde{f},\Psi,T,\epsilon)<\infty.
\end{equation}
From this and Proposition 7.3.13 of \cite{fish19}, for every $\epsilon>0$ small enough it holds
\[ \lim_{T\to \infty} \frac{1}{T}\log \Lambda'(f,\Phi,T,\epsilon)=\lim_{T\to \infty} \frac{1}{T}\log \Lambda(\tilde{f},\Psi,T,\epsilon)=P(\tilde{f},\Psi),\]
where $P(\tilde{f},\Psi)$ is a topological pressure. Since these limits exist and does not depend upon $\epsilon>0$ small enough we deduce that $P'(f,\Phi)$ and $P^*(\tilde{f},\Psi)$ exist as well. From this equation and Equation \eqref{preso} we conclude that
\[ P'(f,\Phi)=P^*(\tilde{f},\Psi)=P(\tilde{f},\Psi)=P(f,\Phi). \]
\end{proof}
Once we know the relation between the asymptotic exponential decrease rates we can show the weighted equidistribution of periodic orbits. 
\begin{pro}
Let $M$ be a closed manifold without conjugate points and with visibility universal covering, $f\in \mc{H}$ be a potential having the weak pressure gap and $\mu$ be the unique equilibrium measure of $f$. Then, $\Phi$-periodic prime orbits weighted by $f$ equidistribute to $\mu$.  
\end{pro}

\begin{proof}
As above, let $\tilde{f}$ be the induced potential on $X$ and $\tilde{\epsilon}>0$ be its Bowen constant given by Proposition \ref{indu}. For every $T>0$ and every $\epsilon<\tilde{\epsilon}$ small enough define the measure
\[ \mu_{T,\epsilon}=\frac{1}{\Lambda'(f,\Phi,T,\epsilon)}\sum_{\gamma\in \mc{O}'_{T,\epsilon}}e^{f(\gamma)}\mu_{\gamma}, \]
where $\Lambda'(f,\Phi,T,\epsilon)<\infty$ by Equation \eqref{expa} and $\mu_{\gamma}$ is the unique $\Phi$-invariant measure supported on $\gamma$. For every $T,\epsilon>0$ we clearly see that $\mu_{T,\epsilon}$ is a $\Phi$-invariant probability measure on $T_1M$.

We will show that $\mu_{T,\epsilon}$ converges to $\mu$. Let $\tau$ be an accumulation point of $(\mu_{T,\epsilon})$ in the weak$^*$ topology. Thus, there exists a sequence $T_n\to \infty$ such that $\mu_{T_n,\epsilon}\to \tau$ weakly. Since $\chi_*$ is weak$^*$ continuous it follows that $\chi_*\mu_{T_n,\epsilon}\to\chi_*\tau$ weakly. Using Equations \eqref{wei} and \eqref{expa} we obtain
\begin{equation}\label{ecis}
    \chi_*\mu_{T_n,\epsilon}=\frac{1}{\Lambda'(f,\Phi,T,\epsilon)}\sum_{\gamma\in \mc{O}'_{T_n,\epsilon}}e^{f(\gamma)}\chi_*\mu_{\gamma}=\frac{1}{\Lambda(\tilde{f},\Psi,T,\epsilon)}\sum_{\tilde{\gamma}\in \mc{O}_{T_n,\epsilon}}e^{\tilde{f}(\tilde{\gamma})}\nu_{\tilde{\gamma}}, 
\end{equation}
where $\nu_{\tilde{\gamma}}=\chi_*\mu_{\gamma}$ is the unique $\Psi$-invariant measure supported on $\tilde{\gamma}=\chi(\gamma)$. 

By Proposition 7.3.15 of \cite{fish19}, we know that $\Psi$-periodic orbits weighted by $\tilde{f}$ equidistribute to the unique equilibrium measure $\nu$ of $\tilde{f}$. Using Equation \eqref{ecis}, this equidistribution means   
\[ \chi_*\mu_{T_n,\epsilon}=\frac{1}{\Lambda(\tilde{f},\Psi,T_n,\epsilon)}\sum_{\tilde{\gamma}\in \mc{O}_{T_n,\epsilon}}e^{\tilde{f}(\tilde{\gamma})}\nu_{\tilde{\gamma}}\longrightarrow \nu \quad \text{ as } \quad n\to \infty \]
in the weak$^*$ topology and hence $\chi_*\tau=\nu=\chi_*\mu$ by Proposition \ref{proy}. By the weak pressure gap, we know that $\mu$ is supported on $\mc{R}_0$ and therefore so is $\tau$:
\[ 1=\mu(\mc{R}_0)=\mu(\chi^{-1}\circ\chi(\mc{R}_0))=\chi_*\mu(\chi(\mc{R}_0))=\chi_*\tau(\chi(\mc{R}_0))=\tau(\chi^{-1}\circ\chi(\mc{R}_0))=\tau(\mc{R}_0). \]
Since $\chi$ is a bijection when restricted to $\mc{R}_0$, $\chi_*\tau=\chi_*\mu$ implies that $\tau=\mu$. As $\tau$ was an arbitrary accumulation point of $(\mu_{T,\epsilon})$ we conclude that $\mu_{T,\epsilon}\to \mu$ weakly, i.e., $\Phi$-periodic prime orbits weighted by $f$ equidistribute to $\mu$.
\end{proof}

\section{Gibbs property}

Let $\Psi=(\psi_t)$ be a continuous flow acting on a compact metric space $X$, $f:X\to\mb{R}$ be a continuous potential and $\nu$ be a $\Psi$-invariant probability measure on $X$. We say that $\nu$ has the Gibbs property for $f$ with constant $P\in\mb{R}$ if for every $\epsilon>0$ there exists $C_{\epsilon}>1$ such that for all $x\in X$ and all $t>0$ it holds
\begin{equation}\label{proper}
    C_{\epsilon}^{-1}\leq \frac{\nu(B(x,\epsilon,t))}{e^{S_tf(x)-tP}}\leq C_{\epsilon},
\end{equation}
where $B(x,\epsilon,t)$ is the $(\epsilon,t)$-dynamical ball with respect to $\Psi$ and $S_tf(x)=\int_0^tg(\psi_s(x))ds$. If $\nu$ is an equilibrium measure of $f$, $\nu$ has the Gibbs property for $f$ with the topological pressure $P(f,\Psi)$ as constant.

Gibbs property was proved for equilibrium measures of Holder continuous potentials in the setting of hyperbolic transitive maps by Bowen \cite{bowenbook}. The extension of this result to continuous expansive flows having the specification property is a consequence of Franco's work \cite{fran77}. We used this result in Lemma \ref{supp} to prove full support of equilibrium measures. In our setting, we show that equilibrium measures of potentials in Theorem \ref{toe}, satisfy the Gibbs property except at non-expansive points.
\begin{lem}\label{ino}
Let $M$ be a closed manifold without conjugate points and with visibility universal covering, $\chi:T_1M\to X$ be the semi-conjugacy between the geodesic and quotient flow and $f\in \mc{H}$. Given $\theta\in T_1M$, $\epsilon>0$ and $t>0$, there exists $\tilde{\epsilon}_1>0$ such that
\[  \chi(B(\theta,\epsilon,t))\subset B(\chi(\theta),\tilde{\epsilon}_1,t). \]
If furthermore $\theta\in \mathcal{R}_0$ then there exist an open saturated set $D(\theta,\epsilon,t)$ and $\tilde{\epsilon}_2>0$ such that
\begin{equation}\label{double}
    B(\chi(\theta),\tilde{\epsilon}_2,t)\subset \chi(D(\theta,\epsilon,t)) \subset \chi(B(\theta,\epsilon,t)),
\end{equation}
where $B(\theta,\epsilon,t)$ and $B(\chi(\theta),\tilde{\epsilon},t)$ are the dynamical balls in $T_1M$ and $X$ respectively.
\end{lem}

\begin{proof}
Recall the notation of Section \ref{c26}: for every $\theta\in T_1M$, $[\theta]=\chi(\theta)$. For the first inclusion, let $\eta\in B(\theta,\epsilon,t)$ hence $d_s(\phi_s(\theta),\phi_s(\eta))<\epsilon$ for all $s\in[0,t]$. By uniform continuity of $\chi$ there exists $\tilde{\epsilon}_1>0$ such that for every $s\in[0,t]$ we have
\[   d(\psi_s[\theta],\psi_s[\eta])<\tilde{\epsilon}_1 \quad \text{ for every }\quad s\in [0,t],   \]
which means that $[\eta]\in B([\theta],\tilde{\epsilon}_1,t)$ and hence $\chi(B(\theta,\epsilon,t))\subset B([\theta],\tilde{\epsilon}_1,t)$.

For the second inclusion, since $B(\theta,\epsilon,t)$ is an open set in $T_1M$, by Lemma \ref{sat} there exists an open saturated set $D(\theta,\epsilon,t)\subset B(\theta,\epsilon,t)$ and hence 
\begin{equation}\label{inclu}
\chi(D(\theta,\epsilon,t))\subset \chi(B(\theta,\epsilon,t)).    
\end{equation}
Since $D(\theta,\epsilon,t)$ is open and saturated it follows that $\chi(D(\theta,\epsilon,t))$ is an open set in $X$. Note that for all $t>0$, $d_t(x,y)=\max_{s\in[0,t]}d(\psi_s(x),\psi_s(y))$ is also a metric on $X$ and its open balls $B_t(x,\epsilon)$ are exactly the dynamical balls $B(x,\epsilon,t)$. We know that the family of dynamical balls $\{B_t(x,1/n)\}_{n\in \mathbb{N}}$ form a basis of neighborhoods for every $x\in X$. Thus, there exists $\tilde{\epsilon}_2>0$ such that $B_t([\theta],\tilde{\epsilon}_2)=B([\theta],\tilde{\epsilon}_2,t)\subset \chi(D(\theta,\epsilon,t))$ because $\chi(D(\theta,\epsilon,t))$ is an open neighborhood of $[\theta]$. Finally, using Equation \eqref{inclu} we obtain $B([\theta],\tilde{\epsilon}_2,t)\subset \chi(B(\theta,\epsilon,t))$.       
\end{proof}

\begin{pro}
Let $M$ be a closed manifold without conjugate points and with visibility universal covering, $f\in \mc{H}$ be a potential having the weak pressure gap and $\mu$ be the unique equilibrium measure of $f$. Then, for every $\epsilon>0$ there exists $C_{\epsilon}>1$ such that for all $\theta\in T_1M$ and all $t>0$ it holds
\[ \mu(B(\theta,\epsilon,t))\leq C_{\epsilon}e^{S_tf(\theta)-tP(f,\Phi)}.  \]
If furthermore $\theta\in \mathcal{R}_0$ then  
\[ \mu(B(\theta,\epsilon,t))\geq C_{\epsilon}^{-1}e^{S_tf(\theta)-tP(f,\Phi)},  \]
where $P(f,\Phi)$ is the topological pressure of $f$ and $S_tf(\theta)=\int_0^tf(\phi_s(\theta))ds$. In particular, $\mu$ has the Gibbs property except at non-expansive points.
\end{pro}

\begin{proof}
Let $\tilde{f}:X\to \mb{R}$ be the induced potential for $f$ (Section \ref{pot}) and $\nu=\chi_*\mu$ be the unique equilibrium measure of $\tilde{f}$. From Lemma \ref{gibs} we know that $\nu$ has the Gibbs property hence for every $\epsilon>0$ there exists $C_{\epsilon}>1$ satisfying Equation \eqref{proper} with $x=[\theta]$ and $P=P(\tilde{f},\Psi)$. Since $A\subset \chi^{-1}\circ\chi(A)$ for any $A\subset T_1M$, from Lemma \ref{ino} it follows that
\begin{equation}
    \mu(B(\theta,\epsilon,t))\leq \mu(\chi^{-1}\circ\chi(B(\theta,\epsilon,t))) =\chi_*\mu(\chi(B(\theta,\epsilon,t)))\leq \nu(B([\theta],\tilde{\epsilon}_1,t)).
\end{equation}
Using this inequality and Equations \eqref{des}, \eqref{pres} and \eqref{proper} we finally obtain 
\[   \mu(B(\theta,\epsilon,t))\leq \nu(B([\theta],\tilde{\epsilon}_1,t)) \leq C_{\epsilon} e^{S_t\tilde{f}[\theta]-tP(\tilde{f},\Psi)}=C_{\epsilon}e^{S_tf(\theta)-tP(f,\Phi)}.   \]
For the second inequality, suppose $\theta\in \mathcal{R}_0$. By Lemma \ref{ino} there exists an open saturated set $D(\theta,\epsilon, t)\subset B(\theta,\epsilon, t)$. Since $D(\theta,\epsilon, t)$ is saturated, Equations \eqref{proper} and \eqref{double} imply that
\begin{equation}
    \mu(D(\theta,\epsilon,t))=\mu(\chi^{-1}\circ \chi(D(\theta,\epsilon, t)))=\chi_*\mu(\chi(D(\theta,\epsilon, t)))\geq \nu(B([\theta],\tilde{\epsilon}_2,t))\geq C_{\epsilon}^{-1}e^{S_t\tilde{f}[\theta]-tP(\tilde{f},\Psi)}.
\end{equation}
Since $S_t\tilde{f}[\theta]=S_tf(\theta)$, $P(\tilde{f},\Psi)=P(f,\Phi)$ and $D(\theta,\epsilon, t)\subset B(\theta,\epsilon, t)$ we conclude that 
\[ \mu(B(\theta,\epsilon,t))\geq \mu(D(\theta,\epsilon,t))\geq C_{\epsilon}^{-1}e^{S_tf(\theta)-tP(f,\Phi)}.\]
\end{proof}

\section{Large deviations principle and entropy density}

Let $\Psi=(\psi_t)$ be a continuous flow acting on a compact metric space $X$. We say that $\Psi$ has entropy density of ergodic measures if for any $\Psi$-invariant measure $\mu$ on $X$ there exists a sequence of ergodic measures $(\mu_n)$ on $X$ such that $\mu_n\to\mu$ in the weak$^*$ topology and $h_{\mu_n}(\psi_1)\to h_{\mu}(\psi_1)$. There is an analogue definition for the discrete case \cite{walt00}.

In the discrete case, Eizenberg-Kifer-Weiss proved entropy density of ergodic measures for $\mb{Z}^d$-shifts having the specification property \cite{eiz94}. Pfister and Sullivan extended this property to maps having an almost global product \cite{fis04}. In the continuous setting, Constantine-Lafont-Thompson showed entropy density for continuous expansive flows having a weak specification property \cite{cons20}. In particular, this holds for all transitive hyperbolic flows including geodesic flows of compact manifolds of negative curvature. Though, for rank-1 surfaces, Coudene and Schapira showed that ergodic measures cannot approximate certain invariant measures supported on flat strips \cite{cou10}. Hence, entropy density of ergodic measures does not hold in our setting. Despite this, it is still possible to approximate invariant measures not supported on strips, thus obtaining a restricted entropy density of ergodic measures.      
\begin{pro}
Let $M$ be a closed manifold without conjugate points and with visibility universal covering, and $\Phi=(\phi_t)$ be its geodesic flow. Given a $\Phi$-invariant measure $\mu$ on $T_1M$ with $\mu(T_1M\setminus\mathcal{R}_0)=0$, there exists a sequence of ergodic measures $(\mu_n)$ on $T_1M$ such that $\mu_n\to \mu$ in the weak$^*$ topology and $h_{\mu_n}(\phi_1)\to h_{\mu}(\phi_1)$. 
\end{pro}
\begin{proof}
Since $\chi$ is a semi-conjugacy it follows that $\chi_*\mu$ is a $\Psi$-invariant measure on $X$. By Proposition \ref{spe}, $\Psi$ is expansive and has the specification property. Proposition 5.5 of \cite{cons20} asserts that $\Psi$ has entropy density of ergodic measures. Thus, there exists a sequence of ergodic measures $(\nu_n)$ on $X$ such that $\nu_n\to \chi_*\mu$ weakly and $h_{\nu_n}(\psi_1)\to h_{\chi_*\mu}(\psi_1)$. By Proposition \ref{exis}, for all $n\geq1$ there exists a $\Phi$-invariant measure $\mu_n$ on $T_1M$ such that $\chi_*\mu_n=\nu_n$ and hence 
\begin{equation}\label{espa}
\mu_n(\mc{R}_0)=\mu_n(\chi^{-1}\circ\chi(\mc{R}_0))=\chi_*\mu_n(\chi(\mc{R}_0))= \nu_n(\chi(\mc{R}_0))\in \{0,1\},
\end{equation}
since $\nu_n$ is ergodic and $\chi(\mc{R}_0)$ is $\Psi$-invariant. We claim that 
\begin{equation}\label{med}
    \mu_n(\mc{R}_0)=1 \quad \text{ for every } \quad n\geq 1.    
\end{equation}
Otherwise $\mu_n(\mc{R}_0)=0$ for all $n\geq 1$. This condition and Equation \eqref{espa} lead to a contradiction:
\[ 1=\mu(\mathcal{R}_0)=\mu(\chi^{-1}\circ\chi(\mathcal{R}_0))= \chi_*\mu(\chi(\mathcal{R}_0))=\lim_{n\to \infty}\nu_n(\chi(\mathcal{R}_0))=\lim_{n\to \infty}\mu_n(\mc{R}_0)=0.  \]
We will show that $\mu_n\to \mu$ weakly. Let $\tau$ be an accumulation point of $(\mu_n)$ in the weak$^*$ topology. Thus there exists a subsequence $(n_k)$ such that $\mu_{n_k}\to \tau$ weakly. Since $\chi_*$ is weak$^*$ continuous it follows that $\nu_{n_k}=\chi_*\mu_{n_k}\to \chi_*\tau$ and hence $\chi_*\tau=\chi_*\mu$. Moreover, $\tau(\mc{R}_0)=\lim_{k\to\infty}\mu_{n_k}(\mc{R}_0)=1$ by Equation \eqref{med}. Since $\chi$ is a bijection when restricted to $\mc{R}_0$ and $\mu(\mc{R}_0)=\tau(\mc{R}_0)=1$, from $\chi_*\tau=\chi_*\mu$ we deduce that $\tau=\mu$. As $\tau$ was an arbitrary accumulation point of $(\mu_n)$ we conclude that $\mu_n\to \mu$ weakly.

We now show the ergodicity of $\mu_n$. Let $A\subset T_1M$ be any $\Phi$-invariant set. Observe that $A=(A\cap\mc{R}_0)\cup (A\cap (T_1M\setminus\mc{R}_0))$. By Equation \eqref{med} we know that $\mu_n(T_1M\setminus \mc{R}_0)=0$ for all $n\geq1$ and hence 
\[ \mu_n(A)=\mu_n(A\cap \mc{R}_0)=\mu_n(\chi^{-1}\circ\chi(A\cap \mc{R}_0))=\chi_*\mu_n(\chi(A\cap \mc{R}_0))=\nu_n(\chi(A\cap \mc{R}_0))\in\{0,1\},\]
by ergodicity of $\nu_n$ and $\Psi$-invariance of $\chi(A\cap \mc{R}_0)$. Therefore, $\mu_n$ is ergodic for all $n\geq 1$. Finally, using entropy density of $\Psi$ and Proposition \ref{sur} we obtain
\[ h_{\mu_n}(\phi_1)=h_{\chi_*\mu_n}(\psi_1)=h_{\nu_n}(\psi_1)\longrightarrow h_{\chi_*\mu}(\psi_1)=h_{\mu}(\phi_1).  \]
\end{proof}
Entropy density of ergodic measures was used by Eizenberg-Kifer-Weiss to prove level-2 large deviations of $\mb{Z}^d$ maps \cite{eiz94}. We can also use this property to show a related result.

Consider a continuous flow $\Psi=(\psi_t)$ acting on a compact metric space $X$. Denote by $\mc{M}_{\Psi}(X)$ the set of $\Psi$-invariant probability measures on $X$. Let $\nu$ be an equilibrium measure of the potential $g$ and $h:X\to \mb{R}$ be a continuous potential. We say that $\nu$ satisfies the level-1 large deviations principle for $h$ if for every $\epsilon>0$ it holds
\begin{equation}\label{lar}
    \lim_{t\to \infty}\frac{1}{t}\log \nu \left\{ x\in X: \left|\frac{1}{t}\int_0^t h(\psi_s(x))ds - \nu(h) \right|\geq \epsilon \right\}= -q(\epsilon) \quad \text{with} 
\end{equation}
\[  q(\epsilon)=  \inf_{\tau\in D_{\epsilon}} P(\Psi,g)-P_{\tau}(\Psi,g) \quad \text{ where }\quad D_{\epsilon}=\{ \nu\in\mc{M}_{\Psi}(X): |\nu(g)-\tau(g)|\geq \epsilon \}. \]
We say that $\nu$ satisfies level-2 large deviations principle if $\nu$ has level-1 large deviations principle for every continuous potential $h$. Level-2 large deviations principle is quite strong and holds for hyperbolic transitive flows \cite{kif90}. In our setting, we show that certain measures having weak pressure gap (Section \ref{uni}) satisfy level-1 large deviations principle for potentials in $\mc{H}$ (Definition \ref{hfam}).
\begin{pro}
Let $M$ be a closed manifold without conjugate points and with visibility universal covering, $f\in \mc{H}$ be a potential having the weak pressure gap and $\mu$ be the unique equilibrium measure of $f$. Then $\mu$ satisfies the level-1 large deviations principle for every $g\in \mc{H}$. 
\end{pro}
\begin{proof}
Let $g\in \mc{H}$ and $\tilde{g}:X\to \mb{R}$ be its induced potential given in Proposition \ref{indu}. From Proposition \ref{proy} we know that $\chi_*\mu=\nu$ is the unique equilibrium measure of the induced potential $\tilde{f}:X\to \mb{R}$. By Proposition \ref{spe}, $\Psi$ is expansive and has the specification property, thus $\Psi$ has the entropy density of ergodic measures by Proposition 5.5 of \cite{cons20}. Applying any of the results of \cite{bom19}, \cite{cons20} or \cite{fis04}, shows that $\nu$ satisfies the level-1 large deviations for every continuous potential $h:X\to \mb{R}$. In particular, $\nu$ and $\tilde{f}$ satisfy Equation \eqref{lar} for $\tilde{g}$. We will verify this equation for $\mu$, $f$ and $g$. For every $t>0$ and $\epsilon>0$ define the sets
\[ \tilde{A}_{t,\epsilon}=\left\{ x\in X: \left|\frac{1}{t}\int_0^t\tilde{g}(\psi_s(x))ds - \nu(\tilde{g}) \right|\geq \epsilon \right\}, A_{t,\epsilon}=\left\{ \eta\in T_1M: \left|\frac{1}{t}\int_0^tg(\phi_s(\eta))ds -\mu(g)\right|\geq \epsilon \right\}. \]
Using the semi-conjugacy Equation \eqref{for} and Proposition \ref{indu}, for every $\eta\in T_1M$ we get
\[  \frac{1}{t}\int_0^t\tilde{g}(\psi_s(\chi(\eta)))ds=\frac{1}{t}\int_0^tg(\phi_s(\eta))ds.   \]
Moreover, since $\chi_*\mu=\nu$, Equation \eqref{des} implies that
\begin{equation}\label{on}
    \nu(\tilde{g})=\int_X\tilde{g}d\nu=\int_{T_1M}gd\mu=\mu(g).
\end{equation}
From these equations and surjectivity of $\chi$ it follows that $\chi(A_{t,\epsilon})=\tilde{A}_{t,\epsilon}$. Now, observe that
\[ A_{t,\epsilon}=(A_{t,\epsilon}\cap \mc{R}_0)\cup (A_{t,\epsilon}\cap(T_1M\setminus \mc{R}_0)), \quad    \tilde{A}_{t,\epsilon}=(\tilde{A}_{t,\epsilon}\cap \chi(\mc{R}_0))\cup (\tilde{A}_{t,\epsilon}\cap(X\setminus \chi(\mc{R}_0))).\]
Since $\mu(T_1M\setminus \mc{R}_0)=\nu(X\setminus\chi(\mc{R}_0))=0$ it follows that $\mu(A_{t,\epsilon})=\mu(A_{t,\epsilon}\cap \mc{R}_0)$ and $\nu(\tilde{A}_{t,\epsilon})=\nu(\tilde{A}_{t,\epsilon}\cap \chi(\mc{R}_0))$. Using these equations we deduce that
\[  \mu(A_{t,\epsilon})=\mu(A_{t,\epsilon}\cap \mc{R}_0)= \mu(\chi^{-1}\circ \chi(A_{t,\epsilon}\cap \mc{R}_0))=\chi_*\mu(\chi(A_{t,\epsilon}\cap \mc{R}_0)))=\nu(\tilde{A}_{t,\epsilon}\cap \chi(\mc{R}_0))=\nu(\tilde{A}_{t,\epsilon}). \]
From this equation and Equation \eqref{lar}, we obtain the desired relation 
\[ \lim_{t\to \infty} \frac{1}{t}\log \mu(A_{t,\epsilon})= \lim_{t\to \infty} \frac{1}{t}\log \nu(\tilde{A}_{t,\epsilon})=-q(\epsilon). \]
We now verify that $q(\epsilon)$ has the desired form. Denote by $\mc{M}_{\Phi}(T_1M)$ the set of $\Phi$-invariant probability measures on $T_1M$. For every $\epsilon>0$ consider the sets
\[  D_{\epsilon}=\{ \tau\in \mc{M}_{\Phi}(T_1M):|\tau(g)-\mu(g)|\geq \epsilon \} \quad\text{ and } \quad  \tilde{D}_{\epsilon}=\{ \omega\in \mc{M}_{\Psi}(X):|\omega(\tilde{g})-\nu(\tilde{g})|\geq \epsilon \}. \]
Using Equation \eqref{on} we see that
\[\mu(g)-\tau(g)=\int_{T_1M}gd\mu-\int_{T_1M}gd\tau=\int_{X}\tilde{g}d\chi_*\mu-\int_X\tilde{g}d\chi_*\tau =\nu(\tilde{g})-\chi_*\tau(\tilde{g}).  \]
From this and surjectivity of $\chi_*$ it follows that $\chi_*(D_{\epsilon})=\tilde{D}_{\epsilon}$. Using this and Equations \eqref{pr} and \eqref{preso} we finally conclude that
\[ q(\epsilon)= \inf_{\omega\in \tilde{D}_{\epsilon}} P(\Psi,\tilde{g})-P_{\omega}(\Psi,\tilde{g})=\inf_{\chi_*\tau \in \chi_*(D_{\epsilon})} P(\Psi,\tilde{g})-P_{\chi_*\tau}(\Psi,\tilde{g})= \inf_{\tau \in D_{\epsilon}} P(\Phi,g)-P_{\tau}(\Phi,g).  \]
\end{proof}

\section{Acknowledgments}
The author appreciates the financial support of CNPQ and FAPEMIG funding agencies during the work. Part of this work was done under the project Math-Amsud of CAPES agency. 

\bibliographystyle{plain}
\bibliography{ref}

\end{document}